\documentclass[copyright,creativecommons]{eptcs}
\providecommand{\event}{GASCom 2024} 

\usepackage[utf8]{inputenc}
\usepackage[english]{babel}
\usepackage{microtype}
\usepackage{url}
\usepackage{amsmath, amssymb, amsthm}
\usepackage{mathtools}

\usepackage[noadjust]{cite}

\usepackage[usenames,dvipsnames,svgnames,table]{xcolor}
\definecolor{darkgreen}{rgb}{0,0.4,0}
\definecolor{BrickRed}{rgb}{0.65,0.08,0}

\theoremstyle{plain}
\newtheorem{theorem}{Theorem}[section]

\newtheorem{proposition}[theorem]{Proposition}

\theoremstyle{remark}

\newtheorem{example}[theorem]{Example}

\newcommand{\oeis}[1]{\href{http://oeis.org/#1}{#1}}

\newcommand{\dd}{\mathtt d}
\newcommand{\uu}{\mathtt u}

\usepackage{tikz}
\usetikzlibrary{positioning}
\usetikzlibrary{arrows}
\tikzstyle{punkt}=[rectangle, rounded corners, draw=black, very thick, text centered]

\newcommand{\addorcid}[1]{\href{https://orcid.org/#1}{\protect\includegraphics[height=3mm]{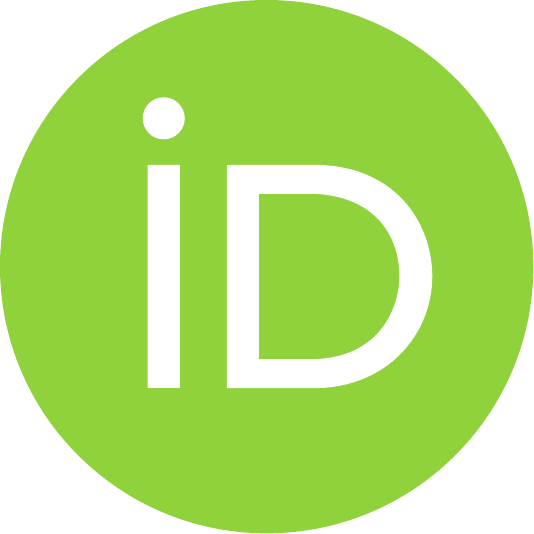}} }

\title{Bijections between Variants of \\ Dyck Paths and Integer Compositions}
\author{%
Manosij Ghosh Dastidar\addorcid{0000-0003-0721-4979}
\institute{TU Wien\\ Wien, Austria}
\institute{Institute of Discrete Mathematics and Geometry}
\email{manosij.dastidar@tuwien.ac.at}
\and 
Michael Wallner\addorcid{0000-0001-8581-449X}
\institute{TU Wien\\ Wien, Austria}
\institute{Institute of Discrete Mathematics and Geometry}
\email{michael.wallner@tuwien.ac.at}
\email{dmg.tuwien.ac.at/mwallner}
}
\def\titlerunning{Bijections between Variants of Dyck Paths and Integer Compositions}
\def\authorrunning{M.\ Ghosh Dastidar \& M.\ Wallner}

\begin{document}

\maketitle

\begin{abstract}
We give bijective results between several variants of lattice paths of length $2n$ (or $2n-2$) and integer compositions of $n$, all enumerated by the seemingly innocuous formula $4^{n-1}$. These associations lead us to make new connections between these objects, such as congruence results. 

\medskip

\end{abstract}


\section{Introduction}

We explore several links between different variants of integer compositions and  generalizations of Dyck paths.
Let us first introduce these objects.
First, an \emph{integer composition} of a nonnegative integer $n$ is a tuple $(n_1,\dots,n_k)$ of nonnegative integers such that $n = n_1 + \dots + n_k$. 
Note that \emph{integer partitions} are integer compositions such that $n_1 \geq n_2 \geq \dots \geq n_k$, (equivalently, the order of summands is not significant). 
Second, a \emph{Dyck path} is a sequence of steps up $\uu =(1,1)$ and down $\dd = (1,-1)$ that starts at the origin, ends on the $x$-axis, and never crosses the $x$-axis. 
All classes of paths we consider will start at the origin and consist of steps $\uu$ and $\dd$, but the constraints will differ.
Natural classes are \emph{Dyck walks} that have no constraints, i.e., they may end anywhere and go below the $x$-axis, 
and \emph{Dyck bridges} (also known as \emph{grand Dyck paths}) that have to end on the $x$-axis but may go below it.

Our results reveal a series of bijections, shown in Figure~\ref{fig:bijections}, connecting these structures through a common enumeration formula $$4^{n-1}.$$ 
This gives the corresponding integer sequence \oeis{A000302} in the OEIS\footnote{The On-Line Encyclopedia of Integer Sequences: \url{http://oeis.org/}} many new combinatorial interpretations.
Additionally, these bijections often map natural statistics onto each other, such as the height of peaks and the number of crossings of the $x$-axis. 

\medskip

\begin{figure}[ht]
\begin{center}
\scalebox{0.52}{%
    \begin{tikzpicture}[
        node distance = 2.5cm and 3.0cm,
        auto,
        ->,
        >=latex',
        font=\LARGE,
        line width=0.8mm,
        block/.style = {circle, draw, minimum size=2.5cm, align=center},
    ]
    
        \node[block] (n1) {Pairs of \\compositions \\of $n$};
        \node[block, right=38mm of n1] (n2) {$3$-compositions \\of $n$\cite{andrews2007theory}};

        \node[block, below=of n1] (n5) {$2$-colored \\Dyck bridges of \\ length $2n-2$};
        \node[block, left=of n5] (n4) {Unconstrained \\Dyck walks \\ of length $2n-2$};
        \node[block, right=of n5] (n7) {Left-to-right max. \\ in Dyck bridges \\ of length $2n$};
        \node[block, right=of n7] (n6) {Dyck paths with \\height-labelled \\peaks of length $2n$};
    
        \draw[<->] (n1) -- node[midway, above] {\hyperref[prop:threecompbijection]{Prop. 2.2}} (n2);
        \draw[<->] (n1) -- node[midway, left] {\hyperref[prop:CompPairsDyckWalk]{Prop. 2.1}} (n4);
        \draw[<->] (n4) -- node[midway, above] {\hyperref[prop:2colunconstrained]{Prop. 2.3}} (n5);
        \draw[<->] (n5) -- node[midway, above] {\hyperref[prop:markedlefttoright2coloredbridge]{Prop. 3.3}} (n7);
        \draw[<->] (n7)  -- node[midway, above] {\hyperref[prop:LabeltoMaxima]{Prop. 3.4}} (n6);
            
+    \end{tikzpicture}}
\end{center}
\vspace{-4mm}
\caption{Bijections proved in this paper of classes of paths and integer compositions that are all enumerated by $4^{n-1}$.}
\label{fig:bijections}
\end{figure}
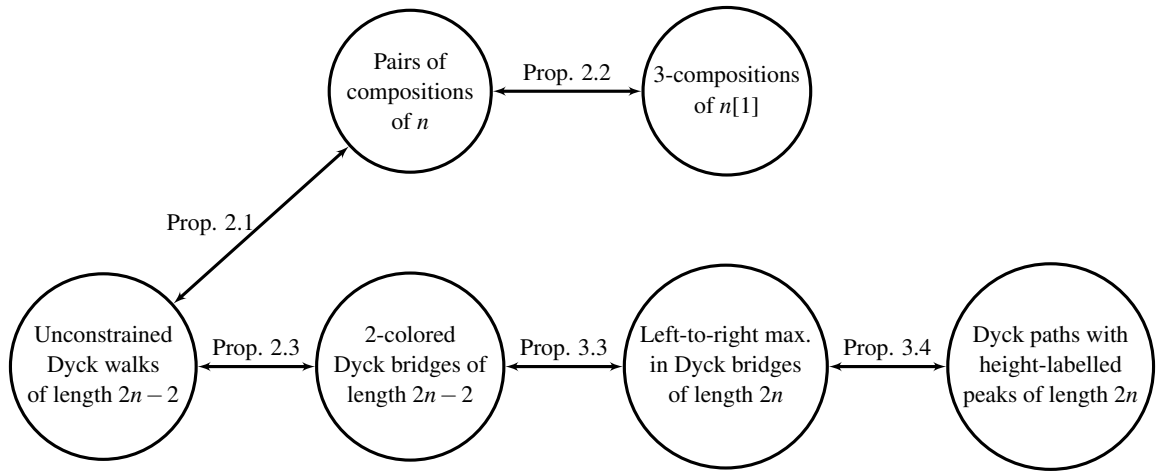

\section{Bijections involving integer compositions}

We start with the simple initial bijection, connecting pairs of compositions and Dyck walks. 
\begin{proposition}
    \label{prop:CompPairsDyckWalk}
    There exists a natural bijection between pairs of compositions of $n$ and Dyck walks of length $2n-2$.
\end{proposition}

\begin{proof}
Let a pair $(A,B)$ of two compositions of $n$ be given.
%
First, we convert each composition to a binary sequence: 
for each element $k$ in a composition, append $k-1$ zeros followed by a~$1$.  
By construction, both of these sequences have to end in $1$. 
So we remove these ones and then concatenate the binary sequences, with $A$'s sequence coming first. 
Finally, after replacing each $0$ by an up step $\uu$ and each $1$ by a down step~$\dd$ the claim follows. 
For the reverse direction cut the walk in the middle into two parts, and re-add the ones.
\end{proof}

A $k$-composition is an integer compositions whose parts come in $k$ different colors with the restriction that the last part of the composition is of the first color; see~\cite{andrews2007theory}. 
We will consider only $3$-compositions.

\begin{proposition}
    \label{prop:threecompbijection}
    There exists a natural bijection between $3$-compositions of $n$ and pairs of compositions of $n$.     
\end{proposition}

\begin{proof}
By definition, the parts of $3$-compositions have three labels $1$, $2$, and $3$.
Anticipating the result, we introduce a notion of left and right:
remove the labels of color~$1$, 
use label $L$ for color~$2$,
and label $R$ for color~$3$.

Now, we describe a map from $3$-compositions of $n$ to pairs of compositions of $n$.
First, we create two identical copies. 
In the first copy, we remove the labels $R$ and add the parts labeled by $L$ to the next part.
If the next part has also a label $L$, then the addition continues to the next part, etc. 
This gives a composition~$A$ without any labels.
Similarly, in the second copy, we remove the labels $L$ and add the parts labeled by~$R$ to the next part.
Again, if the next part has also a label $R$, then the addition continues, 
and we get a composition $B$ without any labels. 
Observe that the size of both compositions has not changed.
Therefore, $(A, B)$ is a pair of compositions of $n$. 

To prove that this map is in fact a  bijection, let us consider an arbitrary pair $(A,B)$ of compositions of $n$. 
The key statistic to consider is the \emph{run of identical parts}:
Let $A=(a_1,a_2,\dots,a_{\ell_A})$ and $B=(b_1,b_2,\dots,b_{\ell_B})$.
A run is a sequence of maximal length such that $a_1=b_1$, $a_2=b_2$, and so on. 
If $a_1 \neq b_2$ we say the run has length $0$. 

For the inverse map, we will describe a recursive algorithm, which reduces the sizes of $A$ and $B$ in the pair $(A,B)$ step-by-step and builds a $3$-composition $C$.
We start with an empty $3$-composition~$C$. 
Depending on the first parts of $A$ and $B$ we distinguish three cases:
\begin{enumerate}    
    \item If $a_1=b_1$ then we attach $a_1$ with label $1$ to $C$, remove $a_1$ from $A$, and remove $b_1$ from $B$.
    \item If $a_1>b_1$ then we attach $b_1$ with color $2$ to $C$, replace $a_1$ in $A$ by $a_1-b_1$, and remove $b_1$ from $B$.
    \item If $a_1<b_1$ then we attach $a_1$ with color $3$ to $C$, remove $a_1$ from $A$, and replace $b_1$ in $B$ by $b_1-a_1$.
\end{enumerate}
We repeat this process with the new values of $A$, $B$, and $C$. 
As the sizes of $A$ and $B$ decrease in each step by at least one, this process terminates. 
Moreover, note that in each step both parts decrease by the same size. 
Hence, in the last step the process ends with case one where both parts are equal, and therefore the final part gets label $1$, as required in the definition of $3$-compositions.
\end{proof}

\medskip

\noindent\begin{minipage}{0.55\textwidth}
\begin{example}
Consider the $3$-composition $6_1+1_2+4_3+2_1$ of $n=13$.
First, we remove the labels of color~$1$, 
use label~$L$ for color~$2$,
and label~$R$ for color~$3$.
Second, we create two identical copies. 
In the first copy, we remove the labels $R$ and add the parts labeled by $L$ to the next part.
If the next part has also a label $L$, then the addition to the next part continues, etc. 
Similarly, in the second copy, we remove the labels $L$ and add the parts labeled by $R$ to the next part.
This gives a pair of compositions of $n$ without any labels, and we have shown in Proposition~\ref{prop:threecompbijection} that this is in fact a bijection. 
\end{example}
\end{minipage}
 \hfill
\begin{minipage}{0.45\textwidth}
\vspace{-0.2\baselineskip}
\begin{center}
\begin{tikzpicture}[node distance=0.75cm, auto, scale=0.8, every node/.style={scale=0.85}]
\node[punkt] (first) at (0,0.1) {\(6_1+1_2+4_3+2_1\)};
\node[punkt] (second) [below=of first] {\((6+1_L+4_R+2)\), \((6+1_L+4_R+2)\)};
\node[punkt] (third) [below=of second] {\((6+1_L+4+2)\), \((6+1+4_R+2)\)};
\node[punkt] (fourth) [below=of third] {\((6,1_L+4,2)\), \((6,1,4_R+2)\)};
\node[punkt] (fifth) [below=of fourth] {\((6,5,2)\), \((6,1,6)\)};

\path[thick,->] 
(first) edge (second)
(second) edge (third)
(third) edge (fourth)
(fourth) edge (fifth);
\end{tikzpicture}
\end{center}
\end{minipage}


\section{Bijections involving Dyck paths}

Let us now consider more complicated classes of Dyck paths. 
All of them use the concept of a \emph{peak}, which is a consecutive pattern $\uu \dd$. 
The first class we consider are \emph{Dyck paths with a marked peak}, which are classical Dyck paths enriched by a marker on a distinguished peak. Two such paths are different, if the underlying paths differ, or, if the paths are the same then two different peaks are marked.
Therefore, the number of these paths is equal to the number of peaks in all Dyck paths, whose enumeration is well-known; see, e.g., \cite[Section 6.1]{Deutsch1999Dyck}.

\begin{theorem}
    \label{theo:Dyckpeak}
    There is an explicit bijection between Dyck paths with a marked peak of height~$h$ and Dyck bridges starting with a $\dd$ step and $h-1$ crossings of the $x$-axis preserving the length.
    Therefore, the number of peaks in all Dyck paths of length $2n$ is equal to $\binom{2n-1}{n}$; see \oeis{A001700}.
\end{theorem}

    \newcommand{\DD}{D}
    \newcommand{\LL}{L}
    \newcommand{\RR}{R}
    
\begin{proof}
    Let $\DD$ be a Dyck path with marked peak at height $h$. 
    Using this peak, we decompose the path $\DD$ into a left part $\LL$ from the origin to this peak and a right part $R$ from this peak to the end: $\DD = \LL \RR$ such that $\LL$ ends with $\uu$ and $\RR$ starts with $\dd$.
    In $\LL$ we perform a last-passage decomposition, cutting at the $\uu$ leaving a certain altitude for the last time;
    while in $\RR$ we perform a first-passage decomposition, cutting at the $\dd$ bringing us down to a new altitude for the first time; see Figure~\ref{fig:firstlastpassageLR}.
    More formally, we have
    \begin{align}
        \label{eq:DyckPeakDecompLR}
        \LL &= \LL_1 \uu \LL_2 \uu \dots \uu \LL_{h} \uu, &
        \RR &= \dd \RR_{h} \dd \RR_{h-1} \dd \dots \dd \RR_{1},
    \end{align}
    where $\LL_i$ and $\RR_i$ for $i=1,\dots,h$ are Dyck paths.
    Now, we pair the paths $\LL_i\uu$ and $\dd \RR_i$ with the same index and map them as follows to non-empty Dyck paths:
    \begin{align*}
        \DD_i = \uu \LL_i \dd \RR_i.
    \end{align*}    
    Then we concatenate these parts, after mapping each part with odd index to its image obtained by any fixed bijection $\varphi$ between Dyck paths and negative Dyck paths. 
    This gives the Dyck bridge
    \begin{align}
        \label{eq:DyckPeakFinalDecomp}
        \varphi(\DD_1) \DD_2 \varphi(\DD_3) \DD_4 \dots \varphi(\DD_{h-1}) \DD_{h}
    \end{align}
    when $h$ is even.
    For odd $h$ it ends with $\varphi(\DD_{h})$.
    This bridge starts with a down  step $\dd$  and crosses the $x$-axis $h-1$ times, as claimed.
    Second, let a Dyck bridge starting with a $\dd$ step be given. 
    We cut at each crossing of the $x$-axis and recover the components $\DD_i$ and $\varphi(\DD_i)$. 
    Hence, it is straightforward to recover the components $\LL_i$ and $\RR_i$ and to rebuild the Dyck path $\DD$ with marked peak. 

Finally, bridges of length $2n$ are counted by $\binom{2n}{n}$, as there is an equal number of up and down steps. Since half of them start with a down step, we get 
\[\frac{1}{2}\binom{2n}{n} = \binom{2n-1}{n}. \qedhere \] 
\end{proof}

   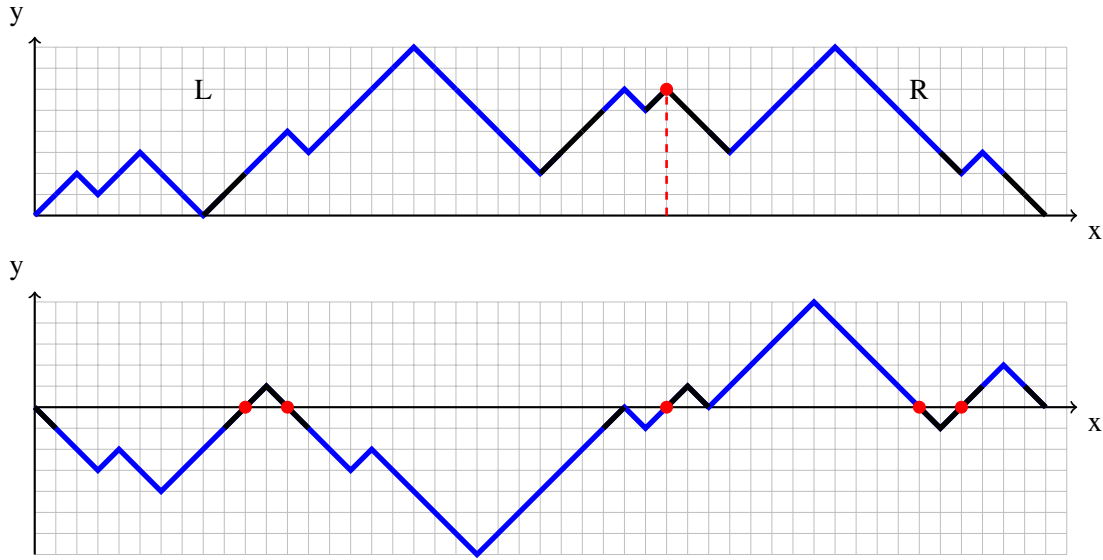
\begin{figure}[t]
   \centering
    \begin{tikzpicture}[scale=0.28]
        
        \newcommand{\pathLength}{49}
        \draw[step=1cm,lightgray,very thin] (0,0) grid (\pathLength,8);
        \draw[thick,->] (0,0) -- (\pathLength+0.5,0) node[anchor=north west] {x};
        \draw[thick,->] (0,0) -- (0,8+0.5) node[anchor=south east] {y};

        \draw[line width=2pt,blue] (0,0) -- ++(1,1); 
        \draw[line width=2pt,blue] (1,1) -- ++(1,1) -- ++(1,-1) -- ++(1,1); 
        \draw[line width=2pt,blue] (4,2) -- ++(1,1) -- ++(1,-1); 
        \draw[line width=2pt,blue] (6,2) -- ++(1,-1); 
        \draw[line width=2pt,blue] (7,1) -- ++(1,-1) -- ++(1,1); 
        \draw[line width=2pt,black] (8,0) -- ++(1,1) -- ++(1,1); 
        \draw[line width=2pt,blue] (10,2) -- ++(1,1); 
        \draw[line width=2pt,blue] (11,3) -- ++(1,1) -- ++(1,-1) -- ++(1,1); 
        \draw[line width=2pt,blue] (14,4) -- ++(1,1); 
        \draw[line width=2pt,blue] (15,5) -- ++(1,1); 
        \draw[line width=2pt,blue] (16,6) -- ++(1,1); 
        \draw[line width=2pt,blue] (17,7) -- ++(1,1) -- ++(1,-1); 
        \draw[line width=2pt,blue] (19,7) -- ++(1,-1); 
        \draw[line width=2pt,blue] (20,6) -- ++(1,-1); 
        \draw[line width=2pt,blue] (21,5) -- ++(1,-1); 
        \draw[line width=2pt,blue] (22,4) -- ++(1,-1); 
        \draw[line width=2pt,blue] (23,3) -- ++(1,-1) -- ++(1,1);
        \draw[line width=2pt,black] (24,2) -- ++(1,1) -- ++(1,1); 
        \draw[line width=2pt,black] (26,4) -- ++(1,1); 
        \draw[line width=2pt,blue] (27,5) -- ++(1,1) -- ++(1,-1) -- ++(1,1); 
        \draw[line width=2pt,black] (29,5) -- ++(1,1) -- ++(1,-1) -- ++(1,-1); 
        \draw[line width=2pt,blue] (32,4) -- ++(1,-1) -- ++(1,1)-- ++(1,1); 
        \draw[line width=2pt,black] (32,4) -- ++(1,-1); 
        \draw[line width=2pt,blue] (35,5) -- ++(1,1); 
        \draw[line width=2pt,blue] (36,6) -- ++(1,1); 
        \draw[line width=2pt,blue] (37,7) -- ++(1,1) -- ++(1,-1); 
        \draw[line width=2pt,blue] (39,7) -- ++(1,-1);
         \draw[line width=2pt,blue] (40,6) -- ++(1,-1); 
        \draw[line width=2pt,blue] (41,5) -- ++(1,-1); 
        \draw[line width=2pt,blue] (42,4) -- ++(1,-1) -- ++(1,-1) -- ++(1,1) -- ++(1,-1); 
        \draw[line width=2pt,black] (43,3) -- ++(1,-1); 
         \draw[line width=2pt,black] (46,2) -- ++(1,-1); 
         \draw[line width=2pt,black] (47,1) -- ++(1,-1); 
        
        \filldraw[red] (30,6) circle (8pt);
        \draw[red, dashed, line width=1pt] (30,0) -- (30,6);

        \node at (8,6) {L};
        \node at (42,6) {R};
    \end{tikzpicture}

      \begin{tikzpicture}[scale=0.28]
    \newcommand{\pathLength}{49}
        \draw[step=1cm,lightgray,very thin] (0,-7) grid (\pathLength,5);
        \draw[thick,->] (0,0) -- (\pathLength+0.5,0) node[anchor=north west] {x};
        \draw[thick,->] (0,-7) -- (0,5+0.5) node[anchor=south east] {y};
    
    \draw[line width=2pt,blue] (0,0)
    -- ++(1,-1)
    -- ++(1,-1)
    -- ++(1,-1)
    -- ++(1,1)
    -- ++(1,-1)
    -- ++(1,-1)
    -- ++(1,1)
    -- ++(1,1)
    -- ++(1,1)
    -- ++(1,1)
    -- ++(1,1)
    -- ++(1,-1)
    -- ++(1,-1)
    -- ++(1,-1)
    -- ++(1,-1)
    -- ++(1,1)
    -- ++(1,-1)
    -- ++(1,-1)
    -- ++(1,-1)
    -- ++(1,-1)
    -- ++(1,-1)
    -- ++(1,1)
    -- ++(1,1)
    -- ++(1,1)
    -- ++(1,1)
    -- ++(1,1)
    -- ++(1,1)
    -- ++(1,1)
    -- ++(1,-1)
    -- ++(1,1)
    -- ++(1,1)
    -- ++(1,-1)
    -- ++(1,1)
    -- ++(1,1)
    -- ++(1,1)
    -- ++(1,1)
    -- ++(1,1)
    -- ++(1,-1)
    -- ++(1,-1)
    -- ++(1,-1)
    -- ++(1,-1)
    -- ++(1,-1)
    -- ++(1,-1)
    -- ++(1,1)
    -- ++(1,1)
    -- ++(1,1)
    -- ++(1,-1)
    -- ++(1,-1)
;

    \draw[line width=2pt,black] (0,0) -- (1,-1); 
    \draw[line width=2pt,black] (9,-1) -- ++(1,1) -- ++(1,1) -- ++(1,-1) -- ++(1,-1); 
    \draw[line width=2pt,black] (27,-1) -- ++(1,1); 
    \draw[line width=2pt,black] (30,0) -- ++(1,1) -- ++(1,-1); 
    \draw[line width=2pt,black] (42,0) -- ++(1,-1) -- ++(1,1) -- ++(1,1); 
    \draw[line width=2pt,black] (47,1) -- ++(1,-1);

    \filldraw[red] (10,0) circle (8pt);
    \filldraw[red] (12,0) circle (8pt);
    \filldraw[red] (30,0) circle (8pt);
    \filldraw[red] (42,0) circle (8pt);
    \filldraw[red] (44,0) circle (8pt);
    
\end{tikzpicture}
    \caption{A Dyck path with a marked peak (red dot) at height 6 and its image under the bijection from Theorem~\ref{theo:Dyckpeak} given by a Dyck bridge starting with a $\dd$~step and $5=6-1$ crossings (red dots). The black steps are used in the last-passage (resp., first-passage) decomposition in the proof. } 
    \label{fig:firstlastpassageLR}
\end{figure}

We return now to the bijections of Figure~\ref{fig:bijections} and
we connect our results with \emph{$2$-colored bridges}; see~\cite[Section 6.4]{BanderierKubaWallner2022}.
They are defined as the concatenation of two bridges such that the first bridge is colored in color $1$ and the second one in color $2$.
Note that contrary to \cite{BanderierKubaWallner2022}, we allow each part to be empty.
Hence, it is easy to see that its generating function is equal to the square of the  generating function $B(z) = \frac{1}{\sqrt{1-4z^2}}$ of bridges:
\[
    B(z)^2 = \frac{1}{1-4z^2}.
\]

\begin{proposition}
    \label{prop:2colunconstrained}
    There is an explicit bijection between $2$-colored Dyck bridges and unconstrained Dyck walks of the same length $2n$.
\end{proposition}

\begin{proof}
    Recall the classical notion of a Dyck meander, which is defined as the prefix of a Dyck path, i.e., a path that starts at $0$, never goes below the $x$-axis, but does not necessarily end on the $x$-axis.
    We will use repeatedly that there is an explicit bijection between Dyck bridges 
    and Dyck meanders of the same length~$2n$; see~\cite{Marchal2003}. 
    Sometimes, it will be necessary to transform a meander further into a negative meander, by flipping all steps, i.e., exchanging $\uu$ by $\dd$ and vice versa. 
    
    We distinguish four cases.    
    First, the first and second bridges are non-empty. 
    The idea is that the change in color corresponds to the last crossing of the $x$-axis.    
    For this purpose we transform the second bridge into a meander or negative meander and 
    attach it to the first bridge such that the attached meander continues on the other side of the $x$-axis. 
    We can easily reverse this procedure by cutting at the last crossing of the $x$-axis.
    All the other cases will have no crossings. 
    Second, if the first bridge is non-empty and the second one is empty, we transform the first bridge into a meander.
    Third, if the first bridge is empty and the second on is non-empty, we transform the second bridge into a negative meander.
    Finally, if both bridges are empty, we map them to the empty walk.    
\end{proof}

We continue, with \emph{Dyck bridges with marked strict left-to-right maximum}. 
A \emph{strict left-to-right maximum} is any peak $\uu\dd$ that has a greater height than all peaks to its left.
We called it marked in the previous sense, when it is attached with a distinguished marker.


\begin{proposition}
    \label{prop:markedlefttoright2coloredbridge}
    There is an explicit bijection between Dyck bridges of length $2n$ with marked strict left-to-right maximum at height $h$ and $2$-colored Dyck bridges of length $2n-2$ with $h-1$ crossings of the $x$-axis in color $1$.
\end{proposition}

\begin{proof}
    Let us start with a Dyck bridge with marked strict left-to-right maximum of length $2n$.
    Then, we cut the bridge at the first return to the $x$-axis after this maximum. 
    The second part to the right is a bridge, which we give color $2$.
    Onto the first part we apply a similar idea as in the bijection of Theorem~\ref{theo:Dyckpeak}.
    As before, we cut the path at the marked left-to-right maximum into a left and right part given by $\LL \RR$, such that $\LL$ ends with $\uu$ and $\RR$ starts with~$\dd$.
    Now, we decompose it similar to~\eqref{eq:DyckPeakDecompLR} into
    \begin{align*}
        \LL &= \varphi(\LL_1) \uu \varphi(\LL_2) \uu \dots \uu \varphi(\LL_{h}) \uu,\\
        \RR &= \dd \RR_{h} \dd \RR_{h-1} \dd \dots \RR_{2} \dd,
    \end{align*}
    where $h$ is the height of the peak, and $\LL_i$ and $\RR_i$ are Dyck paths.
    Note that in this case $\RR$ ends with a $\dd$ step and contains only $h-1$ Dyck paths $\RR_i$.
    As in the proof of Theorem~\ref{theo:Dyckpeak} we form Dyck paths $D_i = \uu \LL_i \dd \RR_i$ for $i=2,\dots,h$.
    Finally, we remove the two steps of the marked peaks, and get the following bridges with two steps less (compare with Equation~\eqref{eq:DyckPeakFinalDecomp}):    
    \begin{align*}
        \varphi(\LL_1) \DD_2 \varphi(\DD_3) \DD_4 \dots \varphi(\DD_{h-1}) \DD_h,
    \end{align*}
    when $h$ is even.
    For odd $h$ it ends with $\varphi(\DD_h)$.
    
    The mapping may be reversed by repeating the aforementioned steps in reverse order.
\end{proof}

We end the bijections shown in Figure~\ref{fig:bijections} by the following link to \emph{Dyck paths with height-labeled peak} that are Dyck paths in which one peak is associated with a label from $\{1,2,\dots,h\}$, where $h$ is the height of the specific peak. 

\pagebreak

\begin{proposition}
\label{prop:LabeltoMaxima}
    There is an explicit bijection between Dyck paths with height-labeled peak with label~$\mu$ at height~$h$ and Dyck bridges of the same length with marked strict left-to-right maximum at height~$\mu$ and $h-\mu$ crossings of the $x$-axis after this maximum.
\end{proposition}

\begin{proof}
    This bijection follows directly from the one described in the proof of Theorem~\ref{theo:Dyckpeak}, whose notation we will use here.
    The difference is that here we concatenate the (positive and negative) Dyck paths differently; see Figure~\ref{fig:bijection-heightlabeled-LRmax}.

    Let a Dyck path with height-labeled peak be given.
    Let $h$ be the height of this peak and $\mu \in \{1,\dots,h\}$ be its label.
    First, we apply the bijection $\varphi$ onto all parts $\LL_i$ in \eqref{eq:DyckPeakDecompLR}.
    From that we get the following bridge in which the height-labeled peak is now a left-to-right maximum (underlined):
    \begin{align*}
        \varphi(\LL_1) \uu \varphi(\LL_2) \uu \dots \uu \varphi(\LL_{h})
        \, \underline{\uu \dd} \,
        \RR_{h} \dd \RR_{h-1} \dd \dots \dd \RR_{1}.
    \end{align*}
    Next, we transform this bridge, such that in the end the height-label $\mu$ constitutes the height of the left-to-right maximum.
    For this purpose, we create and concatenate the paths $\DD_i$ and $\varphi(\DD_i)$ in an alternating fashion at the end:
    \begin{align}
        \label{eq:heightlabeled2LRbridge}
        \varphi(\LL_1) \uu \varphi(\LL_2) \uu \dots \uu \varphi(\LL_{\mu})
        \, \underline{\uu \dd} \,
        \RR_{\mu} \dd \RR_{\mu-1} \dd \dots \dd \RR_{1} 
        \, \varphi(\DD_{\mu+1}) \DD_{\mu+2} \varphi(\DD_{\mu+3}) \dots \DD_h,
    \end{align}
    when $h-\mu$ is even. Otherwise, the last $\DD_h$ is replaced by $\varphi(\DD_h)$.

    For the reverse direction, let a Dyck bridge with marked left-to-right maximum be given. 
    It is then straightforward to decompose it into~\eqref{eq:heightlabeled2LRbridge} and to reverse the steps above to build a Dyck path.
    The left-to-right maximum becomes the height-labeled peak, labeled by its current height. 
    Observe that the height-labeled peak is lifted by the number of crossings of the $x$-axis to the left of this peak.
\end{proof}
\begin{figure}[t]
   \centering
\begin{tikzpicture}[scale=0.28]
    
    \newcommand{\pathLength}{49}
    \draw[step=1cm,lightgray,very thin] (0,0) grid (\pathLength,8);
    \draw[thick,->] (0,0) -- (\pathLength+0.5,0) node[anchor=north west] {x};
    \draw[thick,->] (0,0) -- (0,8+0.5) node[anchor=south east] {y};

    \draw[line width=2pt,blue] (0,0) -- ++(1,1); 
    \draw[line width=2pt,blue] (1,1) -- ++(1,1) -- ++(1,-1) -- ++(1,1); 
    \draw[line width=2pt,blue] (4,2) -- ++(1,1) -- ++(1,-1); 
    \draw[line width=2pt,blue] (6,2) -- ++(1,-1); 
    \draw[line width=2pt,blue] (7,1) -- ++(1,-1) -- ++(1,1); 
    \draw[line width=2pt,black] (8,0) -- ++(1,1) -- ++(1,1); 
    \draw[line width=2pt,blue] (10,2) -- ++(1,1); 
    \draw[line width=2pt,blue] (11,3) -- ++(1,1) -- ++(1,-1) -- ++(1,1); 
    \draw[line width=2pt,blue] (14,4) -- ++(1,1); 
    \draw[line width=2pt,blue] (15,5) -- ++(1,1); 
    \draw[line width=2pt,blue] (16,6) -- ++(1,1); 
    \draw[line width=2pt,blue] (17,7) -- ++(1,1) -- ++(1,-1); 
    \draw[line width=2pt,blue] (19,7) -- ++(1,-1); 
    \draw[line width=2pt,blue] (20,6) -- ++(1,-1); 
    \draw[line width=2pt,blue] (21,5) -- ++(1,-1); 
    \draw[line width=2pt,blue] (22,4) -- ++(1,-1); 
    \draw[line width=2pt,blue] (23,3) -- ++(1,-1) -- ++(1,1);
    \draw[line width=2pt,black] (24,2) -- ++(1,1) -- ++(1,1); 
    \draw[line width=2pt,black] (26,4) -- ++(1,1); 
    \draw[line width=2pt,blue] (27,5) -- ++(1,1) -- ++(1,-1) -- ++(1,1); 
    \draw[line width=2pt,black] (29,5) -- ++(1,1) -- ++(1,-1) -- ++(1,-1); 
    \draw[line width=2pt,blue] (32,4) -- ++(1,-1) -- ++(1,1)-- ++(1,1); 
    \draw[line width=2pt,black] (32,4) -- ++(1,-1); 
    \draw[line width=2pt,blue] (35,5) -- ++(1,1); 
    \draw[line width=2pt,blue] (36,6) -- ++(1,1); 
    \draw[line width=2pt,blue] (37,7) -- ++(1,1) -- ++(1,-1); 
    \draw[line width=2pt,blue] (39,7) -- ++(1,-1);
     \draw[line width=2pt,blue] (40,6) -- ++(1,-1); 
    \draw[line width=2pt,blue] (41,5) -- ++(1,-1); 
    \draw[line width=2pt,blue] (42,4) -- ++(1,-1) -- ++(1,-1) -- ++(1,1) -- ++(1,-1); 
    \draw[line width=2pt,black] (43,3) -- ++(1,-1); 
     \draw[line width=2pt,black] (46,2) -- ++(1,-1); 
     \draw[line width=2pt,black] (47,1) -- ++(1,-1); 
    
    \filldraw[red] (30,4) circle (8pt);
    \draw[red, dashed, line width=1pt] (30,0) -- (30,6);

    \node at (8,6) {L};
    \node at (42,6) {R};
\end{tikzpicture}
    
\begin{tikzpicture}[scale=0.28]

    \newcommand{\pathLength}{49}
    \draw[step=1cm,lightgray,very thin] (0,-4) grid (\pathLength,8);
    \draw[thick,->] (0,0) -- (\pathLength+0.5,0) node[anchor=north west] {x};
    \draw[thick,->] (0,-4) -- (0,8+0.5) node[anchor=south east] {y};
 
    \coordinate (current) at (0,0);

    \draw[line width=2pt,blue] (0,0) -- ++(1,-1) 
    -- ++(1,-1) 
    -- ++(1,1) 
    -- ++(1,-1) 
    -- ++(1,-1) 
    -- ++(1,1) 
    -- ++(1,1) 
    -- ++(1,1) 
    -- ++(1,1) 
    -- ++(1,1) 
    -- ++(1,-1) 
    -- ++(1,-1) 
    -- ++(1,1) 
    -- ++(1,-1) 
    -- ++(1,-1) 
    -- ++(1,-1) 
    -- ++(1,-1) 
    -- ++(1,-1) 
    -- ++(1,1) 
    -- ++(1,1) 
    -- ++(1,1) 
    -- ++(1,1) 
    -- ++(1,1) 
    -- ++(1,1) 
    -- ++(1,1) 
    -- ++(1,1) 
    -- ++(1,-1) 
    -- ++(1,1) 
    -- ++(1,1) 
    -- ++(1,1) 
    -- ++(1,1) 
    -- ++(1,1) 
    -- ++(1,-1) 
    -- ++(1,-1) 
    -- ++(1,-1) 
    -- ++(1,-1) 
    -- ++(1,-1) 
    -- ++(1,-1) 
    -- ++(1, 1) 
    -- ++(1,-1) 
    -- ++(1,-1) 
    -- ++(1,-1) 
    -- ++(1,-1) 
    -- ++(1,-1) 
    -- ++(1,1) 
    -- ++(1,1) 
    -- ++(1, 1) 
    -- ++(1,-1); 
    
    \draw[line width=2pt,black] (8,0) -- ++(1,1) -- ++(1,1); 
    \draw[line width=2pt,black] (24,2) -- ++(1,1) -- ++(1,1) -- ++(1,-1);
    \draw[line width=2pt,black] (37,3) -- ++(1,-1); 
    \draw[line width=2pt,black] (40,2) -- ++(1,-1) -- ++(1,-1);  

     \filldraw[red] (26,4) circle (8pt);
     \filldraw[red] ([xshift=-8pt,yshift=-8pt]42,0) rectangle ++(16pt,16pt);
     \filldraw[red] ([xshift=-8pt,yshift=-8pt]46,0) rectangle ++(16pt,16pt);
     
\end{tikzpicture}
    
    \caption{A Dyck path with a height-labeled peak (red dot) with label $4$ at height $6$ and its image under the bijection from Proposition~\ref{prop:LabeltoMaxima} given by a Dyck bridge with marked strict left-to-right maximum (red dot) at height $4$ and $2=6-4$ crossings (red squares). 
    } 
    \label{fig:bijection-heightlabeled-LRmax}
    
\end{figure}
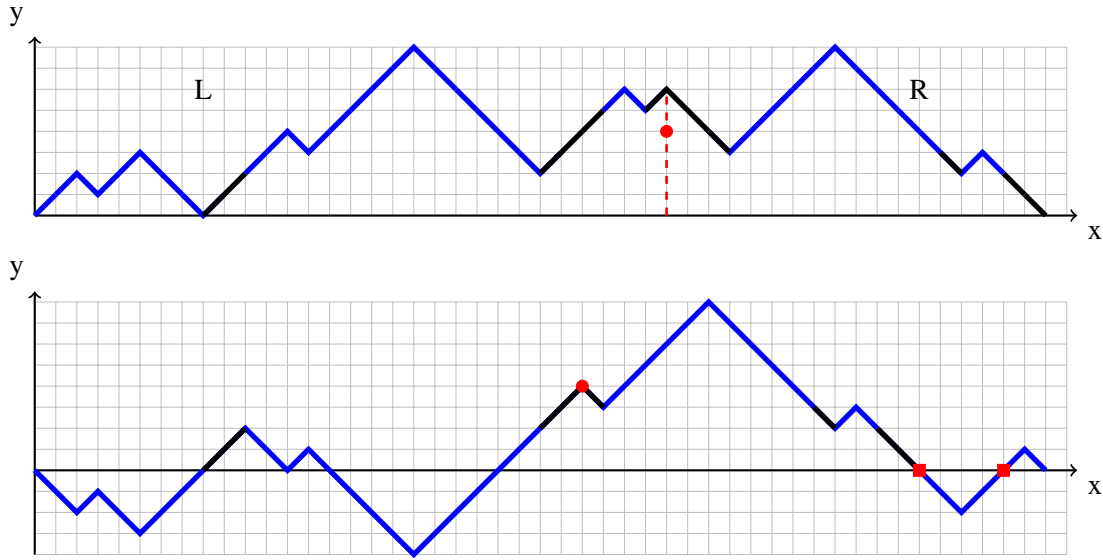

\section{Conclusion and Outlook}

Having established links between the internal structures of Dyck paths and integer compositions, it is only natural to ask whether important theorems from one subject can be transported to the other. When we talk about integer compositions or partitions we are keen to see arithmetic properties in those structures. 
In the long version of this work~\cite{GhoshDastidarWallner2024Bijections}, we give further bijective links and we show that such arithmetic results also exist in lattice paths. 
In particular, we are pleased to note that: 

\begin{theorem}[{\cite[Theorem~3.8]{GhoshDastidarWallner2024Bijections}}]
Let $D_r(n)$ be the number of Dyck paths with semi-length $n$ and with \emph{exactly} $r$ peaks for every reached height. 
Then $D_r(n) \equiv 0 \pmod{r+1}$ for $n>r$.
\end{theorem}

In the opposite direction we also want to see if important theorems in integer compositions and partitions can ``generate'' theorems in the world of lattice paths. In 2020, Kim, Kim, and Lovejoy~\cite{zbMATH07247777} observed the phenomenon of \emph{parity bias} in partitions, where they showed that: 
if $p_o(n)$ denotes the number of partitions of $n$ with more odd parts than even parts and if $p_e(n)$ denotes the number of partitions of $n$ with more even parts than odd parts, then $p_o(n) > p_e(n)$. 

In a subsequent article we will show that the analogous theorem is true even for integer compositions.
Furthermore, we will demonstrate that even for Dyck paths a similar result holds, when segregating paths with respect to whether they have more peaks at odd or even heights.


\bibliographystyle{eptcs}
\bibliography{bibliography}   

\begin{thebibliography}{1}
\providecommand{\bibitemdeclare}[2]{}
\providecommand{\surnamestart}{}
\providecommand{\surnameend}{}
\providecommand{\urlprefix}{Available at }
\providecommand{\url}[1]{\texttt{#1}}
\providecommand{\href}[2]{\texttt{#2}}
\providecommand{\urlalt}[2]{\href{#1}{#2}}
\providecommand{\doi}[1]{doi:\urlalt{https://doi.org/#1}{#1}}
\providecommand{\eprint}[1]{arXiv:\urlalt{https://arxiv.org/abs/#1}{#1}}
\providecommand{\bibinfo}[2]{#2}

\bibitemdeclare{article}{andrews2007theory}
\bibitem{andrews2007theory}
\bibinfo{author}{George~E. \surnamestart Andrews\surnameend}
  (\bibinfo{year}{2007}): \emph{\bibinfo{title}{The {T}heory of {C}ompositions,
  {IV}: {M}ulticompositions.}}
\newblock {\slshape \bibinfo{journal}{The Mathematics student}}
  \bibinfo{volume}{76}, p.~\bibinfo{pages}{25}.

\bibitemdeclare{article}{BanderierKubaWallner2022}
\bibitem{BanderierKubaWallner2022}
\bibinfo{author}{Cyril \surnamestart Banderier\surnameend},
  \bibinfo{author}{Markus \surnamestart Kuba\surnameend} \&
  \bibinfo{author}{Michael \surnamestart Wallner\surnameend}
  (\bibinfo{year}{2024}): \emph{\bibinfo{title}{Phase transitions of
  composition schemes: {M}ittag-{L}effler and mixed {P}oisson distributions}}.
\newblock {\slshape \bibinfo{journal}{To appear in Ann. Appl. Probab.}}
\newblock \eprint{2103.03751}.

\bibitemdeclare{article}{Deutsch1999Dyck}
\bibitem{Deutsch1999Dyck}
\bibinfo{author}{Emeric \surnamestart Deutsch\surnameend}
  (\bibinfo{year}{1999}): \emph{\bibinfo{title}{Dyck path enumeration}}.
\newblock {\slshape \bibinfo{journal}{Discrete Math.}}
  \bibinfo{volume}{204}(\bibinfo{number}{1-3}), pp. \bibinfo{pages}{167--202},
  \doi{10.1016/S0012-365X(98)00371-9}.

\bibitemdeclare{article}{GhoshDastidarWallner2024Bijections}
\bibitem{GhoshDastidarWallner2024Bijections}
\bibinfo{author}{Manosij \surnamestart Ghosh~Dastidar\surnameend} \&
  \bibinfo{author}{Michael \surnamestart Wallner\surnameend}
  (\bibinfo{year}{2024}): \emph{\bibinfo{title}{Bijections and congruences
  involving lattice paths and integer compositions}}.
\newblock \eprint{2402.17849}.

\bibitemdeclare{article}{zbMATH07247777}
\bibitem{zbMATH07247777}
\bibinfo{author}{Byungchan \surnamestart Kim\surnameend},
  \bibinfo{author}{Eunmi \surnamestart Kim\surnameend} \&
  \bibinfo{author}{Jeremy \surnamestart Lovejoy\surnameend}
  (\bibinfo{year}{2020}): \emph{\bibinfo{title}{Parity bias in partitions}}.
\newblock {\slshape \bibinfo{journal}{Eur. J. Comb.}} \bibinfo{volume}{89},
  p.~\bibinfo{pages}{18}, \doi{10.1016/j.ejc.2020.103159}.
\newblock \bibinfo{note}{Id/No 103159}.

\bibitemdeclare{incollection}{Marchal2003}
\bibitem{Marchal2003}
\bibinfo{author}{Philippe \surnamestart Marchal\surnameend}
  (\bibinfo{year}{2003}): \emph{\bibinfo{title}{{Constructing a sequence of
  random walks strongly converging to {B}rownian motion}}}.
\newblock In: {\slshape \bibinfo{booktitle}{Proceedings of Discrete Random
  Walks (DRW'03)}}, \bibinfo{publisher}{Discr. Math. Theor. Comput. Sci.}, pp.
  \bibinfo{pages}{181--190}, \doi{10.46298/dmtcs.3335}.

\end{thebibliography}

\end{document}